\documentclass[12pt]{article}
\usepackage{amssymb,amsmath,amsfonts,amsthm,amsxtra,setspace}

\newcommand{\mf}[1]{\mathfrak{#1}}

\newcommand{\mc}{\mathcal}
\newcommand{\bsigma}{\mathbf{\mathop{\pmb{\sum}}}}
\newcommand{\set}[1]{\{#1\}}

\newcommand{\PPI}{\mathbf{PPI}}
\newcommand{\CW}{\mathbf{CW}}
\newcommand{\forces}{\Vdash}
\newcommand{\noopsort}[1]{}
\newtheorem{prob}{Problem}

\newtheorem{claim}{Claim}

\newtheorem{theorem}{Theorem}[section]

\newenvironment{rem}{\begin{trivlist} \item[] {\bf Remark.}}{\hspace*{0pt}\end{trivlist}}

\newsavebox{\Prfref}

\newsavebox{\prfref}

\newtheoremstyle{ref}
{\topsep}	
{\topsep}	
{\it}
{}
{}
{}
{ }
{\thmname{{\bfseries#1}}\thmnumber{ \textbf{#2\thmnote{\rm #3}\textbf .}}}

\theoremstyle{ref}
\newtheorem{lem}[theorem]{Lemma}
\newtheorem{thm}[theorem]{Theorem}
\newtheorem{prop}[theorem]{Proposition}

\newtheorem{sthm}[theorem]{Theorem}
\newtheorem{slem}[theorem]{Lemma}

\newtheorem{scor}[theorem]{Corollary}
\newtheorem{definition}[theorem]{Definition}

\newtheorem{lemma}[theorem]{Lemma}

\theoremstyle{nnref}
\newtheorem*{defn}{Definition}

\begin{document}
\title{PFA$(S)[S]$ and countably compact spaces}
\author{Alan Dow{$^1$} and Franklin D. Tall{$^2$}}

\footnotetext[1]{Research supported by NSF grant DMS-1501506.}
\footnotetext[2]{Research supported by NSERC grant A-7354.\vspace*{2pt}}
\date{\today}
\maketitle

\begin{abstract}
We show a number of undecidable assertions concerning countably
compact spaces hold under  
PFA(S)[S]. We also show the
consistency without large cardinals of \emph{every 
locally compact, perfectly normal space is paracompact}.
\end{abstract}

\renewcommand{\thefootnote}{}
\footnote
{\parbox[1.8em]{\linewidth}{$2000$ Math.\ Subj.\ Class.\ Primary 54A35, 03E35, 54D20; Secondary 03E55, 03E65, 03E75, 54A20, 54D15, 54D30, 54D45,
54D55.}\vspace*{5pt}}
\renewcommand{\thefootnote}{}
\footnote
{\parbox[1.8em]{\linewidth}{Key words and phrases: locally compact, normal, countably compact, free sequence, countably tight, PFA$(S)[S]$, perfect pre-image of $\omega_1$, $\omega$-bounded, paracompact, sequentially compact.}}

\section{Introduction}

This note is a sequel to \cite{DT}. 
As shown in  \cite{LTo}, \cite{To}, \cite{LT},  and \cite{DT},
forcing with a coherent Souslin tree  over a model of an iteration
axiom such as PFA produces a model with many of the consequences  
of the iteration axiom plus some useful consequences of V = L. 
As in \cite{To}, it is useful to catalog the former consequences for future
use, especially when their proofs are non-trivial. As in [5], it is
also  useful to distinguish consequences of the method which do not
require large cardinals.
 For a  discussion of what we call PFA(S)[S],
see e.g. \cite{LT} and \cite{D1}.
 Our main results here are a proof that
PFA(S)[S] implies countably tight perfect pre-images of $\omega_1$
include copies of  $\omega_1$, a simplification of the proof in \cite{LT} that
PFA(S)[S] implies locally compact, perfectly  normal spaces are
paracompact, and a demonstration that this last conclusion can be
obtained  without the use of large cardinals.
  
In \cite{DT} we proved

\begin{sthm}\label{thmMain2}
PFA$(S)[S]$ implies every sequentially compact, non-compact regular space includes an uncountable free sequence.  If the
space has character $\leq \aleph_1$, then it includes a copy of $\omega_1$.
\end{sthm}

\begin{scor}
PFA$(S)[S]$ implies $(\PPI)$: Every first countable perfect pre-image of $\omega_1$ includes a copy of $\omega_1$.
\end{scor}

Although this paper is about PFA$(S)[S]$, many of the results follow
from just PFA. Some are new, because, although $\PPI$ was known to
follow from PFA, the stronger version in Theorem \ref{thmMain2} was
not formulated and proven from PFA earlier. A simpler version of the
proof in \cite{DT} proves it from PFA. 

\section{Countably compact, perfectly normal spaces}
\begin{sthm}\label{thmClosedGDelta}
PFA$(S)[S]$ implies every countably compact regular space with closed sets $G_\delta$'s is compact.
\end{sthm}

The conclusion was proved from MA$_{\omega_1}$ by W. Weiss \cite{W}.  It is not at all obvious that it can be
obtained from PFA$(S)[S]$, since Weiss' proof uses MA$_{\omega_1}(\sigma$-centred$)$ essentially.  The perfectly normal version of Theorem \ref{thmClosedGDelta} was proved by S. Todorcevic several years ago with a non-trivial stand-alone proof.  We now can easily obtain the stronger version from \ref{thmMain2}.

\begin{proof}
Since open sets are $F_\sigma$'s, discrete subspaces of such a space are $\sigma$-closed-discrete and hence countable.  Countably compact regular spaces with points $G_\delta$ are first countable and hence sequentially compact, so the result follows from Theorem \ref{thmMain2}.
\end{proof}

In \cite{LT} it was shown that

\begin{thm}\label{thm:paracompact}
There is a model of PFA$(S)[S]$ in which every locally compact perfectly normal space is paracompact.
\end{thm}

The proof depended on $\bsigma^-$:
\begin{quotation}
\noindent \textit{In a compact $T_2$, countably tight space, locally countable subspaces of size $\aleph_1$ are $\sigma$-discrete.}
\end{quotation}

The proof of $\bsigma^-$ (see \cite{FTT}) depended on the following result, due to Todorcevic:

\begin{slem}[{~\cite{To}}]\label{lemTodorcevic}
PFA$(S)[S]$ implies every compact countably tight space is sequential.
\end{slem}

Lemma \ref{lemTodorcevic} has taken a very long time to appear.  It
therefore may be of interest that we can avoid using it to prove
Theorem \ref{thm:paracompact}.  We first note that in \cite{FTT},
$\bsigma^-$ is proved for compact sequential spaces, and Lemma
\ref{lemTodorcevic} is then appealed to.  Let us therefore show that
that restricted version of $\bsigma^-$ suffices. 

\begin{thm}\label{thmpnpc}
Assume:
\begin{enumerate}
\item{
Countably compact, perfectly normal spaces are compact;}
\item{
$\bsigma^-$ for compact sequential spaces;
}
\item{
Normal first countable spaces are collectionwise Hausdorff.
}
\end{enumerate}
Then locally compact, perfectly normal spaces are paracompact.
\end{thm}

\begin{proof}
Given a locally compact, perfectly normal space $X$ that is not
compact, consider its one-point compactification $X^* = X \cup
\set{*}$.  To show $X^*$ is sequential, take a non-closed subspace $Y$
of $X^*$.  If $Y \subseteq X$, $Y$ is not countably compact.  Then it
has a countably infinite closed discrete subspace $D$.  Viewed as a
sequence, $D$ converges to the point at infinity.  Now suppose $*$,
the point at infinity, is in $Y \subseteq X^*$.  Let $z$ be a limit
point of $Y$ which is not in $Y$.  Then $z$ is a limit point of $Y -
\set{*}$, and so there is a sequence from $Y$ converging to $z$, since
$X$ is first countable. 
\end{proof}

We now can show the closure of a Lindel\"of subspace $Z$ of $X$ is
Lindel\"of.  $\overline{Z}^*$, the one-point compactification of
$\overline{Z}$, is sequential.  We claim it is hereditarily  
Lindel\"of.  If not, it has a right-separated subspace $R$ of size
$\aleph_1$, which by $\bsigma^-$ for sequential spaces is
$\sigma$-discrete.  Let $R'$ be an uncountable discrete subspace of $R
\cap \overline{Z}$.  By closed sets $G_\delta$, $R'$ is
$\sigma$-closed-discrete, so let $R''$ be a closed discrete subspace
of $R$ of size $\aleph_1$.  By $\aleph_1$-collectionwise Hausdorffness
and normality, expand $R''$ to a discrete collection of open sets.
The traces of those open sets on $Z$ form an uncountable discrete
collection, contradicting Lindel\"ofness.\hfill\qed\\ 

Continuing with the proof of Theorem \ref{thmpnpc}, we have improved
\cite{LT} by not needing that there are no first countable
$L$-spaces. 
Instead we use:

\begin{lem}
Locally compact, perfectly normal, collectionwise Hausdorff spaces are
the topological sum of subspaces with Lindel\"of number $\leq \aleph_1$.
\end{lem}

This was proved by G. Gruenhage in \cite{G2}, with the slightly stronger assumption of
``collectionwise normality with respect to compact sets". This was later
improved by H.J.K. Junnila (unpublished) to just use collectionwise Hausdorffness.
  Since a sum of paracompact spaces is paracompact, we have reduced the problem to showing:

\begin{quotation}
\noindent ($\dagger$) \textit{Locally compact, perfectly normal spaces with Lindel\"of number $\leq \aleph_1$ are paracompact}.
\end{quotation}

Because such a space has countable tightness and -- under our
assumptions -- has closures of Lindel\"of subspaces Lindel\"of, it can
be written as an increasing union of $\{X_\alpha\}_{\alpha <
  \omega_1}$, where $X_\alpha$ is Lindel\"of and open,
$\overline{X}_\alpha \subseteq X_{\alpha+1}$, and for limit $\lambda$,
$X_\lambda = \bigcup_{\alpha < \lambda} X_\alpha$.  If the space were
not paracompact, stationarily often $\overline{X}_\alpha - X_\alpha$
would be non-empty.  Picking a point from each such boundary, one
obtains a locally countable subspace of size $\aleph_1$.  Applying
$\bsigma^-$ for sequential spaces and closed sets $G_\delta$, that
subspace is $\sigma$-closed-discrete.  Via pressing down, normality,
and $\aleph_1$-collectionwise Hausdorffness, we obtain an uncountable
discrete collection of open sets tracing onto some $X_\alpha$,
contradicting Lindel\"ofness. \hfill\qed 

Instead of relying on Junnila's unpublished work, we could have used:

\begin{lem}
Suppose every locally compact, perfectly normal space is
collectionwise Hausdorff. Then every locally compact, perfectly normal
space is collectionwise normal with respect to compact sets. 
\end{lem}

\begin{proof}
Let $\mc{K}$ be a discrete collection of compact sets in a locally
compact, perfectly normal space $X$. Consider the quotient space $X /
\sim$ obtained by collapsing each member of $\mc{K}$ to a point. It is
routine to verify that $X / \sim$ is locally compact, perfectly
normal, and that if $X / \sim$ is collectionwise Hausdorff, then
$\mc{K}$ is separated in $X$.                                                          
\end{proof}

Following the scheme in \cite{DT}, we shall also prove:

\begin{thm}
If ZFC is consistent then so is ZFC plus ``locally compact perfectly normal spaces are paracompact".
\end{thm}
\begin{proof}
First, as in \cite{LT}, we perform a preliminary forcing to obtain
$\diamondsuit$ for stationary systems on all regular uncountable
cardinals.  This will give us a coherent Souslin tree $S$, the form of
$\diamondsuit$ on $\omega_2$ used in \cite{DT}, and after we force
with an $\aleph_2$-c.c. proper $S$-preserving iteration of size
$\aleph_2$, that 
\begin{quotation}
\noindent ($*$) \textit{normal first countable $\aleph_1$-collectionwise Hausdorff spaces are collectionwise Hausdorff}.
\end{quotation}
Following the scheme in \cite{DT} (previously seen in \cite{To1} and
\cite{D}) we form an $\aleph_2$-length countable support iteration of
$\aleph_2$-p.i.c. $S$-preserving proper posets 
establishing $\PPI$ and $\bsigma^-$, after which we then force with
$S$.  The last forcing was shown in \cite{LT} to produce: 
\begin{quotation}
\noindent ($\CW$) \textit{normal first countable spaces are $\aleph_1$-collectionwise Hausdorff}.
\end{quotation}
As noted in \cite{LT}, the other forcings will preserve $\diamondsuit$
for stationary systems on $\omega_2$, so combining ($*$) and $\CW$, we
obtain (3). 

The only new point to observe is that each stage of the iteration
producing $\bsigma^-$ for compact sequential spaces is proper and
$S$-preserving, and has cardinality $\aleph_1$, by CH.  ``Proper and
$S$-preserving" was shown in \cite{FTT}.  As in the $\PPI$ without
large cardinals proof in \cite{DT}, we can code up the necessary
information about locally countable sets of size $\aleph_1$ so as to
get a poset of size $\aleph_1$ for establishing an instance of
$\bsigma^-$.  We can then iterate $\aleph_2$ times, alternately
forcing with the $\PPI$ and $\bsigma^-$ posets, before forcing with
$S$.  Since, as noted in \cite{DT}, $\bsigma^-$ implies $\mathfrak{b}
> \aleph_1$, we don't need to add dominating reals to obtain this as
in \cite{DT}.  As was done in \cite{DT}, one uses the diamond to show
that the iteration is long enough. 
\end{proof}

If one is not trying to avoid Moore-Mr\'owka (i.e. the conclusion of \ref{lemTodorcevic}), it is not hard to prove:

\begin{thm}
Assume $\bsigma^-$ and that normal first countable spaces are collectionwise Hausdorff.  Then locally compact, perfectly normal spaces are paracompact.
\end{thm}

\begin{proof}
We only used ``countably compact perfectly normal spaces are compact"
in order to prove that the space's one-point compactification was
sequential.  Instead we shall prove just that it is countably tight.
By \cite{B}, it suffices to prove our space includes no perfect
pre-image of $\omega_1$.  But that follows easily from perfect
normality. 
\end{proof}

Here is a variation of Theorem \ref{thmClosedGDelta}.

\begin{sthm}\label{thm27}
PFA$(S)[S]$ implies that if $X$ is countably compact, locally compact,
and does not include a perfect pre-image of $\omega_1$, then $X$ is
compact. 
\end{sthm}
\begin{proof}
By \cite{B}, the one-point compactification $X^*$ is countably tight.
By PFA$(S)[S]$ and \cite{To} (also \ref{lemTodorcevic}), $X^*$ is
sequential and hence 
sequentially compact.  (Dow \cite{D1} obtains sequential compactness
via PFA$(S)[S]$ without using \ref{lemTodorcevic}.)  Then $X$ is
sequentially compact, for take a sequence.  It has a convergent
subsequence in $X^*$.  The limit of that sequence cannot be the point
at infinity, $*$, else the sequence would be closed discrete,
violating countable compactness.  By \ref{thmMain2} it suffices to
show $X$ has no uncountable free sequence.  Suppose it did.  Then that
free sequence has $*$ in its closure, else it would be free in $X^*$,
contradicting countable tightness.  Since the closure of the free
sequence in $X$ is not closed in $X^*$, it is not sequentially closed
there, so there is a sequence in the closure of the free sequence
which converges to $*$, since there is nowhere else for it to converge
to.  But then, again there is an infinite closed discrete subset of
$X$, contradiction. 
\end{proof}

The following corollary is of interest with regard to the question of whether there exist large countably compact, locally countable spaces.

\begin{scor}\label{cor28}
PFA$(S)[S]$ implies that if $X$ is uncountable, countably compact, locally countable, and $T_3$, then $X$ includes a copy of $\omega_1$.
\end{scor}
\begin{proof}
Since $X$ is locally countable, it cannot be compact.  It is locally compact, so by the Theorem, it must include a perfect pre-image of $\omega_1$.  It is first countable, so then must include a copy of $\omega_1$.
\end{proof}

\section{$\omega$-bounded spaces}
Here are some more applications of PFA$(S)[S]$ to countably compact spaces.  Recall:

\begin{defn}
A space is \textbf{$\mathbf{\omega}$-bounded} if every countable subset of it has compact closure.
\end{defn}
It is known that:
\begin{lem}[~{\cite{G}}]\label{lem:boundedPreimage}
An $\omega$-bounded space which is not compact includes a perfect pre-image of $\omega_1$.
\end{lem}

For hereditarily normal spaces, under PFA$(S)[S]$ we can improve ``perfect pre-image" to copy:

\begin{sthm}\label{thm:omegaBoundedCopy}
PFA$(S)[S]$ implies every hereditarily normal, $\omega$-bounded, non-compact space includes a copy of $\omega_1$.
\end{sthm}

\begin{proof}
Apply Lemma \ref{lem:boundedPreimage}, $\PPI$, and
Lemma \ref{lem:dt1}.
\end{proof}

\begin{lem}[~{\cite{DT1}}]
PFA$(S)[S]$ implies every\label{lem:dt1}  
hereditarily normal perfect pre-image of $\omega_1$ includes a first
countable perfect pre-image of $\omega_1$ and hence a copy of
$\omega_1$.  
\end{lem}

Surprisingly, we can do considerably better:

\begin{sthm}\label{thm43}
PFA$(S)[S]$ implies every hereditarily normal, countably compact, non-compact space includes a copy of $\omega_1$.
\end{sthm}
\begin{proof}
The first step is to prove:

\begin{sthm}\label{thm44}
PFA$(S)[S]$ implies separable, hereditarily normal, countably compact spaces are compact.
\end{sthm}

\begin{proof}
Let $\dot X$ be our $S$-name of a
hereditarily normal separable countably compact
space. We first explain  that it suffices to show 
that $\dot X$ is forced to be sequentially compact. 

The argument is by contradiction and goes as follows. If $\dot X$ is forced  to be sequentially compact but not compact, then  by Theorem \ref{thmMain2} $\dot X$ is forced to contain an uncountable free sequence. We then invoke \cite{Ny3} which tells us that under $\mf{q} = \aleph_1$ -- and hence under $\CW$ -- separable, hereditarily normal, countably compact spaces do not have uncountable free sequences.

Now we show that $\dot X$ is sequentially compact. 
Suppose that $\omega$ is any infinite discrete
subset of $\dot X$,
and assume that $\omega$ is forced to have no converging subsequence. 

If we consult the proof of Theorem 3.3 in \cite{D1} we find in the proof that there must be a condition $s\in S$, an infinite subset
 $b\subset \omega$, 
and a family  $F_s$ of functions from $\omega$ into $[0,1]$ such
that $s$ forces each $f\in F_s$ has a continuous extension to all of
$\dot X$, and for each infinite $a\subset b$, there is an $f\in F_s$
such that $s$ forces that $f[a]$ is not a converging sequence in
$[0,1]$.  Since the proof is short we provide it here.
 \bigskip
 
   For each $t\in S$,   let $F_t$ denote the
    set of all $f\in [0,1]^\omega$ such that $t$ forces that $f$ has a
    continuous extension to all of $X$.  Say that $a\subset \omega$ is
    split by $F_t$ if there is an $f\in F_t$ such that $t$
    forces that the set $\{
    f(n) : n\in a\}$ does not converge.
 Fix any well-ordering of $S$ in order
      type $\omega_1$ and recursively choose,
      if possible,  a mod
      finite, length $\omega_1$  chain of infinite subsets $\{ a_t : t\in S\}$ 
      of $\omega$ so that  $a_t$ is not split by $F_t$. 
  We are working in PFA(S) (a model 
of $\mathfrak p
      =\mathfrak c$) so we may choose    an $a$ 
       which is mod finite contained in $a_t$ for all $t\in S$. It
       is easy to see that 1 forces that $a$ is a converging sequence in
       the countably compact space $X$.  Therefore, 
       this induction must have stopped at some $s\in S$
       and if we choose any infinite $a_s$ mod finite contained in $a_t$
       for each $t$ coming before $s$ in the well-ordering,
  we have have the desired   pair $b, F_s $.

Now $F_s$ is a family of actual functions, not just names of
functions. This means there is an embedding $e$ of $\omega$ into
$[0,1]^{F_s}$ where $e(n) = e_n$ is defined by $e_n(f) = f(n)$. 
It follows that $\{ e_n : n\in b\}$ is a completely divergent sequence
in $[0,1]^{F_s}$. 
This of course means that the closure $K$ of $\{e_n: n\in b\}$ is a
compact non-sequential space. 
This is happening in a model of PFA(S), and we prove in \cite{DT} that there is an uncountable free sequence in $K$. 
It follows easily that there is a $\aleph_1$-sized subset $F_s'$ 
of $F_s$ satisfying that $\{ e_n \restriction F_s' : n\in b\}$ 
also has an uncountable free sequence in its closure. 
Let $\{ y_\alpha  : \alpha \in \omega_1\} \subset [0,1]^{F_s'}$
 denote this free sequence. Since the PFA(S) model is a model of
 $\mathfrak p>\omega_1$, we may fix, for each $\alpha\in \omega_1$, 
 a subset $a_\alpha$ of $b$ so that $\{ e_n \restriction F_s'
 : n\in a_\alpha\}$ converges to $y_\alpha$.

Now we pass to the PFA(S)[S] extension
and have a look at $X$. Clearly the product of the
 family of continuous extensions of
the functions in $F_s'$, call this $\varphi$,
 is a continuous function  from $X$ into $[0,1]^{F_s'}$.
For each $\alpha \in \omega_1$, let $x_\alpha$
be any limit point of $a_\alpha$.
For each $n\in b$, $\varphi(n)$ is equal to $x_n\restriction F_s'$,
and therefore, $\varphi(x_\alpha)$ is equal to $y_\alpha$.
It follows immediately that
 $\{ y_\alpha : \alpha\in \omega_1\}$ is
a free sequence.
\end{proof}

Theorem \ref{thm44} gives:

\begin{scor}
PFA$(S)[S]$ implies hereditarily normal, countably compact spaces are $\omega$-bounded.
\end{scor}
We then apply \ref{thm:omegaBoundedCopy} to obtain \ref{thm43}.
\end{proof}

\begin{rem}
Under PFA, one can replace the ``hereditarily" in Theorem \ref{thm44}
by ``first countable", but not under PFA$(S)[S]$.  PFA$(S)[S]$ implies
$\mathfrak{p} = \aleph_1$ \cite{To} and under that hypothesis, there
is a locally compact, locally countable, separable, normal, countably
compact space which is not compact.  This space is due to Franklin and
Rajagopalan \cite{Fr}.  See discussion in Section 7 of \cite{vD} and
Section 2 of \cite{N}. 
\end{rem}

The PFA version of this next result was proven by
 Eisworth \cite{eisworth2002}.

\begin{sthm}\label{thmpreimagecopy}
PFA$(S)[S]$ implies a countably tight perfect pre-image of $\omega_1$
includes a copy of $\omega_1$. 
\end{sthm}

\begin{proof}
By Lemma \ref{lemTodorcevic} (as shown in \cite{To}) 
 PFA$(S)[S]$ implies that all countably tight
compact spaces are sequential. Since a perfect pre-image of $\omega_1$
is locally compact  it will suffice to prove
this theorem for sequential perfect pre-images of $\omega_1$.
The proof is technical and requires familiarity with \cite{DT} so is
postponed to the end of the paper. 
\end{proof}

Theorem \ref{thmpreimagecopy} will be used in \cite{DT1} to prove:

\begin{prop}
	There is a model of form PFA$(S)[S]$ in which a locally compact, normal, countably tight space is paracompact if and only if its separable closed subspaces are Lindel\"of, and it does not include a copy of $\omega_1$.
\end{prop}

Just as for hereditarily normal, one can vary \ref{lem:boundedPreimage} to get

\begin{sthm}
PFA$(S)[S]$ implies a countably tight, $\omega$-bounded space which is not compact includes a copy of $\omega_1$.
\end{sthm}
\begin{proof}
Immediate from \ref{lem:boundedPreimage} and \ref{thmpreimagecopy}.
\end{proof}

\begin{prob}
Does PFA$(S)[S]$ imply every non-compact, countably tight, countably
compact space includes a perfect pre-image of $\omega_1$?  If so, by
\ref{thmpreimagecopy} it would include a copy of $\omega_1$. 
\end{prob}

\section{Some problems of Nyikos}
In \cite{Ny3} Peter Nyikos raises a number of questions about
hereditarily normal countably compact spaces and settles some of them
under PFA.  We can obtain similar results under PFA$(S)[S]$ and also
answer some questions he left open.  We shall use his numbering, for
easy reference. \\ 

\noindent \textbf{Statement A}.  \textit{Every compact space of
  countable tightness is sequential}.  \\ Follows from PFA$(S)[S]$
\cite{To}. \\ 

\noindent \textbf{Statement 2} (= 1.3(2)).  \textit{Every separable, $T_5$, countably compact space is compact}.  \\ Follows from PFA$(S)[S]$: Theorem \ref{thm44}. \\

\noindent \textbf{Statement 3}  \textit{Every countably compact $T_5$ space is either compact or includes a copy of $\omega_1$}.  \\ Follows from PFA$(S)[S]$: Theorem \ref{thm43}. \\

\noindent \textbf{1.3(1)} \textit{Every free sequence in a separable, countably compact $T_5$ space is countable}.  \\ Follows from $\mf{q} = \aleph_1$, and hence from PFA$(S)[S]$.\\

\noindent \textbf{1.3(3)} \textit{Every countably compact $T_5$ space is sequentially compact}. \\ Follows from PFA$(S)[S]$.
\begin{proof}
The proof of \ref{thm44} establishes this in the separable case.  Apply the separable case to the closure of a given sequence.
\end{proof}

We thus have PFA$(S)[S]$ implies \\

\noindent \textbf{Statement 1} (1.3(4)) \textit{Every compact $T_5$ space is sequentially compact}.\\

\noindent \textbf{1.4} \textit{In a countably compact $T_5$ space, every countable subset has compact, Fr\'{e}chet-Urysohn closure}.

We can do better under PFA$(S)[S]$.  We actually get first countable
closure.  By \ref{thm44} we get compact; by \cite{To} we get
hereditarily Lindel\"of, hence first countable, since it is shown
there that PFA$(S)[S]$ implies compact, hereditarily normal, separable
spaces are hereditarily Lindel\"of. 

Nyikos' Problem 2 asks if Statement A is compatible with $\mf{q} = \aleph_1$.  It is.

He asks whether the following is consistent: \\

\noindent \textbf{Statement 5}.  \textit{Every compact separable $T_5$ space is of character $< \mf{p}$}. \\ 

Since PFA$(S)[S]$ implies such spaces are first countable, the answer
is trivially ``yes". 

Nyikos' Problem 5 asks if there is a ZFC example of a separable,
$T_5$, locally compact space of cardinality $\aleph_1$.  He points out
that if there are no locally compact $S$-spaces and $\mf{q} =
\aleph_1$, then there is a negative answer.  Since the one-point
compactification of an $S$-space is an $S$-space, PFA$(S)[S]$ yields a
negative answer. 

\section{Countably tight perfect pre-images of $\omega_1$} 
We now complete the proof of Theorem \ref{thmpreimagecopy}. 
The reader
is referred to  \cite{DT} for the final stages of the proof.
As noted above in the first steps of the proof of Theorem
\ref{thmpreimagecopy}, we may assume,
for the remainder of the section, that we have an
 $S$-name $\dot X$ of a 
sequential space with a perfect  mapping $\dot f$
onto the ordinal $\omega_1$. Note that $\dot X$ is forced
to be locally compact.
By passing to a subspace
 we can assume that the preimage of each successor ordinal is a
 singleton. Then, by simple renaming, 
we may assume that $\{ \alpha+1 : \alpha \in \omega_1\}$ 
is a dense subset of $\dot X$.
We can now 
assume that the base set for $\dot X$ is the ordinal $\omega_2$
 (since PFA(S) implies that $\mathfrak c = \omega_2$, and
   the cardinality of a sequential space of density at most $\omega_1$
   is at most  $\mathfrak c$).
Next we fix 
  an assignment of $S$-names of open neighborhood bases
 $\{ \dot U(x,\xi) : \xi \in \omega_2\}$, for each $x\in \omega_2$. 
Obviously repetitions are allowed.
 We may assume, by the continuity of $\dot f$,
 that 1 forces that 
 $\dot f[\dot U(x,\xi) ] \subset [0, \dot f(x)]$
for all $x,\xi \in \omega_2$. 

Now we discuss the special forcing
properties that a  coherent Souslin tree will have.
 Assume that $g$ is a  generic filter on $S$ viewed as a
cofinal branch.  For each $s\in S$, $o(s)$ is  the level
(order-type of domain) of $s$ in  $S$. For any $t\in S$, define
$s \oplus t$ to be the function $s\cup t\restriction[o(s),o(t))$. 
Of course when $o(t) \leq o(s)$, $s\oplus t$ is simply $s$.
 One of the
properties of $S$ ensures that $s \oplus t\in S$ for all $s,t\in S$. We
similarly define $s \oplus g$ to be the branch $\{ s \oplus t : t\in
g\}$.
 We let $\dot X[g]$ (or even $X[g]$)
denote the space obtained by evaluating the topology
 using $g$,  so $\dot X[s\oplus g]$ will be a
 different perfect pre-image of $\omega_1$ also existing in the model
 $V[g]$.  More generally for an $S$-name $\dot A$, we will
let $\dot A[g]$ denote the standard evaluation of $\dot A$ by $g$.
\bigskip

\newcommand{\rk}{\operatorname{rk}}

We recall a useful method of calculating closure in
sequential spaces.
A tree $T \subset \omega^{<\omega}$ is said to be
well-founded if it contains no  no infinite branch. 
 Let $\mathbf{WF}$ denote all  downward closed
well-founded  trees $T\subset \omega^{<\omega}$ with the property
that each non-maximal node has a full  set of immediate successors.
Such a tree has an associated rank function, $\rk_T $ which maps
elements of $T$ into $\omega_1$.  If $t\in T$ is a maximal node, 
then $\rk_t(t) = 0$, and otherwise,
$\rk_T(t)$ is equal to  $\sup \{ \rk_T(t')+1  : t < t'\in T_t\}$.
The rank of $T$ itself will be $\rk_T(\emptyset)$
and 
we let $\mathbf{WF}(\alpha)$ denote the set of trees
of rank less than $\alpha$. 
Suppose that $\sigma$ is any function from 
$\max(T) = \{t\in T: \rk_T(t)=0\}$ into  $\omega_1$ as
a subset of a space $X$. 
By induction on rank of $t\in T$, define an evaluation
 $e(\sigma,t)$ to be the limit (if it exists) 
of the sequence $\{ e(\sigma,t^\frown n ) : n\in \omega\}$. We will
say that $\sigma$ is $X$-converging if $e(\sigma,t)$ exists for all
$t\in T$. It is well-known that every point in the sequential closure
of $\omega_1\subset X$ 
will equal $e(\sigma,\emptyset)$ for some such $\sigma$.

Now, given our $S$-name $\dot X$, we will define $\Lambda$ to be the
set of all $\sigma$ as above (the underlying 
 $T\in \mathbf{WF}$ is simply the downward closure of the
 domain of $\sigma$) with the propery that $1$ forces that $\sigma$
 is converging.  
Since $S$ has the ccc property and adds no new subsets of $\omega$, 
 it  follows that if $s\in S$ forces that an ordinal
 $\zeta\in\omega_2$ is in the (sequential) closure of some
 $\delta$, then there is a $\sigma\in \Lambda$ such that $s$ forces
 that $e(\sigma,\emptyset) = \zeta$. However, even though $1$ forces
 that $\sigma$ converges, it is not true that $1$ forces that
 $e(\sigma,\emptyset) =\zeta$. Nevertheless the set $\Lambda$ makes
 for a very useful substitute for $S$-names of members of $\dot X$. 
Now that the actual ordinal value associated to $\lambda$ depends on
the generic, we will use $e_g(\sigma)$
 (and suppress the second
coordinate) to refer to the ordinal $\zeta$ in $X[g]$
that is equal to 
 $e(\sigma,\emptyset)$. Similarly, we use
 $e_s(\sigma)$ if $s$ decides this value.

\newcommand{\dom}{\operatorname{dom}}

\begin{definition} For any sequence 
$\langle \sigma_n : n\in  \omega\rangle$  
of members of $\Lambda$, 
say that $\sigma$ is constructed from 
$\langle \sigma_n : n\in  \omega\rangle$  
if for each $n\in \omega$ and 
each node $t\in \dom(\sigma)$ (a maximal node of the associated
tree) with $t(0)=n$, and each node $t_1\in \dom(\sigma_n)$,
\begin{enumerate}
\item there a node $t_2\in\dom(\sigma)$ such that
 $t_2(0)=n$, $\sigma(t_2) = \sigma_n(t_1)$, $\dom(t_2)=1+\dom(t_1)$,
 and
$t_2(1+j) = t_1(j)$ for all $j\in \dom(t_1)$,
\item there is a node $t_3\in \dom(\sigma_n)$ such that 
$\sigma(t) =\sigma_n(t_3)$,  $\dom(t)=1+\dom(t_3)$, and
$t_3(j)=t(1+j) $ for all $j\in \dom(t_3)$,
\end{enumerate}
 \end{definition}

\begin{definition}
 For\label{seql}
 each integer $n>0$, and subset $B$ of $\Lambda^n$ we  define
 the hierarchy $\{ B^{(\alpha)} : \alpha\in \omega_1\}$
by recursion.   
 For a limit $\alpha$, $B^{(\alpha)}$
 (which could also be denoted as $B^{(<\alpha)}$) will equal
  $\bigcup_{\beta<\alpha} B^{(\beta)}$. The members of $B^{(\alpha+1)}$
  for any $\alpha$, will consist of  $B^{(\alpha)}$ together with all
$\vec b\in \Lambda^n$ with the property that there is a sequence
 $\langle \vec b_\ell : \ell\in \omega\rangle$ consisting of members
 of $B^{(\alpha)}$ such that, for each $i<n$, $\vec b(i)$ is built
 from the sequence $\{ \vec b_\ell(i) : \ell\in \omega\}$.

A subset $B$ of $\Lambda^n$ will be said to be $S$-sequentially closed
if $B^{(\omega_1)} = B$.
\end{definition}

The next lemma should be obvious.

\begin{lemma} For each $A\subset \Lambda$,
 $1$ forces that $e[
A^{(\omega_1)}]$
  is a sequentially compact subset of $\dot X$.
\end{lemma}

\begin{definition} For each $S$-name $\dot A$ and $s\forces \dot A\subset
 \Lambda^n$,  we define 
the $S$-name  $(\dot A)^{(\omega_1)}$ according to the
  property that for each $s<t$ and $t\forces \vec b\in (\dot
  A)^{(<\omega_1)}$, there is a countable $B\subset Y^n$ such that 
 $t\forces B\subset \dot A$ and $\vec b\in B^{(<\omega_1)}$. 
\end{definition}

By our assumption that $\omega_1$ has no complete accumulation points,
the family $\{ (\omega_1\setminus \delta)^{(\omega_1)} : \delta\in
\omega_1\}$ 
is a free filter of $S$-sequentially  closed subsets of $\Lambda$.
By Zorn's
Lemma, we 
can extend it to a  maximal free filter, $\mathcal F_0$,
of  $S$-sequentially closed subsets of $\Lambda$. 
To apply PFA(S) we require that we have a maximal filter
in the forcing extension by $S$.
The filter $\mathcal F_0$ may not generate a maximal filter in the
extension  $V[g]$ and so we will have to extend it. 
We will also need there to be a close connection between
 the behavior of our chosen maximal filter 
 in $V[g]$ and in $V[s\oplus g]$ for all $s\in S$. We
 refer to this as ``symmetry''.

We introduce some notational conventions.  Let $S^{<\omega}$ denote
the set of finite tuples 
$\langle s_i : i< n\rangle$  (ordered lexicographically) for which
there is a 
$\delta$ such that each $s_i\in S_\delta$. Our convention will be that
they are distinct elements. We let $bS$ denote the collection
 $\{ s\oplus g : s\in S\}$ (technically this is an $S$-name for the
set of all $\omega_1$-branches of $S$ in $V[g]$). We will be
working in the product structure $\Lambda^{bS}$ and
we let $\Pi_{\langle s_i : i<n\rangle}$ denote the
projection from $\Lambda^{bS}$ to 
the product $\Lambda^{ \{ s_i \oplus g~ :~ i< n\}}$,
and we let ${\tilde {\Pi}}_{\langle s_i ~:~ i<n\rangle}$ be the projection
onto $\Lambda^n$.  The notation $\Pi_{\langle s_i ~:~ i< n\rangle}^X$ will be
used as the notation for the projection (in the extension 
 $V[g]$) from
the product space
 $\mathbf{\Pi}\{ X[{s\oplus g}] ~:~ s\in S\}$ 
(we ignore
repetitions) onto $\mathbf{\Pi} \{ X[{s_i\oplus g}] ~:~ i<n\}$.
In case of possible confusion, we adopt a standard convention
that for a singleton $s\in S$,  we identify $\Lambda^{\{s\}}$ with
$\Lambda$ 
and $\mathbf{\Pi}\{ X[{s\oplus g}] \}$ with $X[{s\oplus g}]$; and
similarly $\Lambda^n$ with $n=1$ is treated as simply being $\Lambda$.

\begin{definition}
For $\vec \sigma \in \Lambda^{\langle s_i ~:~ i<n\rangle}$ and
generic\label{notation} 
 $g$, we intend that $e_g(\vec \sigma)$ should equal the vector
 $\langle e_{s_i\oplus g}(\vec \sigma_i) ~:~ i<n\rangle$.
Similarly, for $\dot A$ a name of a subset of $\Lambda^{\langle s_i ~:~
  i<   n\rangle}$,  
if a condition $s$ forces a value on $\dot A$, then
 we can use $e_s(\dot A)$ as an abbreviation for 
 $e_{s\oplus g}(\dot A) = \{ e_{s\oplus g}(\vec \sigma) ~:~ 
\vec \sigma \in \dot A[s\oplus g]\}$.
\end{definition}

\begin{definition} Suppose that  $\dot A$ is an $S$-name of a subset
  of $\Lambda^n$ for some $n$, in particular, that  some $s$ forces
  this. Let  
 $s'$ be any other member of $S$ with $o(s')=o(s)$. We define a new
 name $\dot A^{s}_{s'}$ (the $(s,s')$-transfer perhaps) which is
 defined by the property that for all $\langle \sigma_i \rangle_{i<n}\in
 \Lambda^n$ and  $s<t\in S$ such that $t\forces \langle \sigma_i \rangle_{i<n}
 \in \dot A$, we have that $s' \oplus t \forces \langle \sigma_i
 \rangle_{i<n} \in \dot A^s_{s'}$. 
\end{definition}

\begin{lemma} For any generic $g\subset S$,
 $ \dot A[{s \oplus g}] =  (\dot A^s_{s'})[{s'\oplus g}]$.  
\end{lemma} 

\begin{theorem} There is a  family\label{filterF} 
$\mathcal F$ = 
$\{ 
(s^\alpha,\{s^\alpha_i ~:~ i<n_\alpha\}, \dot F_\alpha ) ~:~ \alpha \in
  \lambda\}$ where, 
\begin{enumerate}
\item for each $\alpha\in \lambda$, $
\{ s^\alpha_i ~:~  i<
  n_\alpha\}\in S^{<\omega}$, $s^\alpha\in S$, $ o(s^\alpha_0)\leq o(s^\alpha)$,  
\item $\dot F_\alpha$ is an $S$-name such that $s^\alpha\Vdash \dot
  F_\alpha = (\dot F_\alpha)^{(\omega_1)} \subset \Lambda^{n_\alpha}$ 
\item for each $s\in S$ and $F\in \mathcal F_0$, $(s,\{s\}, \check F)\in 
 \mathcal F$,
\item for each $s\in S_{o(s^\alpha)}$, $(s,\{s^\alpha_i ~:~i<
  n_\alpha\}, (\dot F_\alpha)^{s^\alpha}_{s})\in \mathcal F$,
\item for each generic\label{previousitem}
 $g\subset S$, the family
$ \{ \tilde\Pi_{\langle s^\alpha_i ~:~ i<n_\alpha\rangle}^{-1}( 
(\dot F_\alpha)[g]) ~:~ 
  s^\alpha \in g\}$ is finitely directed;
we let $\dot {\mathcal  F}_1$ be the $S$-name for the filter base it
generates.
\item
For each generic $g\subset S$ and each $\langle  s_i ~:~ i< n\rangle\in 
 S^{<\omega}$, the family \\
  $\{  (\dot F_\alpha)[g] ~:~ s^\alpha\in g \ \mbox{and}\ 
 \{ s_i \oplus g ~:~ i< n\} = \{ s^\alpha_i \oplus g ~:~ i<n_\alpha\} \}$ is
 a maximal filter on the family of $S$-sequentially closed subsets of
 $\Lambda^n$.
\end{enumerate}
\end{theorem}

\begin{proof} 
Straightforward recursion or Zorn's Lemma argument over the family 
 of ``symmetric'' filters (those satisfying (1)-(5)).
\end{proof}

\begin{definition}
 For any $\langle s_i ~:~ i<\ell\rangle\in S^{<\omega}$, let 
  $\dot {\mathcal F}_{\langle s_i~:~i<\ell\rangle}$, 
 respectively   $\dot {\mathcal F}^{\sim}_{\langle s_i~:~i<\ell\rangle}$, 
 denote the filter on $\Lambda^{\langle s_i \oplus g~:~ i<\ell\rangle}$,
 respectively 
    $\Lambda^{\ell}$, induced by 
$ \Pi_{\langle s_i ~:~      i<\ell\rangle}(\dot{\mathcal F}_1)$,
 respectively
$ \tilde\Pi_{\langle s_i ~:~      i<\ell\rangle}(\dot{\mathcal F}_1)$.
Except for re-naming of the index set, these are the same.
\end{definition}

\begin{definition} Let $\mathcal A$ denote the family\label{needA}
 of all
$(s,\langle s_i ~:~ i< \ell\rangle, \dot A)$ satisfying that  $o(s)\geq o(s_0)$, 
$\langle s_i ~:~
i< \ell\rangle \in S^{<\omega}$, and 
 $s\Vdash 
 \dot A \in \dot
 {\mathcal F}_{\langle s_i~:~i<\ell\rangle}^+$. As usual, 
for a family $\mathcal G$ of sets,
 $\mathcal  G^+$ denotes the family of sets that meet each member of
 $\mathcal G$.
\end{definition}

\begin{lemma} For each $(s,\langle s_i~:~ i<n\rangle, \dot A)\in \mathcal A$, 
the object  $(s,\langle s_i~:~i<n\rangle , \dot
A^{(\omega_1)})$ is in the list $\mathcal F$.
\end{lemma}

\subsection{$S$-preserving proper forcing}

This next statement was a lemma in the  proof of PPI 
 in \cite{DT},  we change it to a definition.

\begin{definition} Suppose that\label{acceptable}
 $M\prec H(\kappa)$ (for suitably big $\kappa$) is a
  countable elementary submodel
  containing $\Lambda,  \mathcal
  A$. Let $M\cap \omega_1 = \delta$.
Say that the sequence
$\langle  y^M(s) ~:~ s\in S_\delta\rangle$ 
is an $M$-acceptable sequence providing
that for every $(\bar s,\{s_i ~:~ i<n\}, \dot A)\in \mathcal A\cap M$,
and every $s \in S_\delta$ with $\bar s < s$, there is a
 $B\subset \Lambda^n\cap M$ such that $s\forces B\subset \dot A$
and $s \forces \langle  y^M(s_i \oplus s) ~:~ i< n\rangle \in
B^{(\delta+1)}$. 
\end{definition}

We must give a new proof that there is an acceptable sequence
 consisting of points with a special countable sequence of
neighborhoods (see Lemma \ref{firstcountable}).  The complicated
condition (\ref{three}) is capturing the net effect of all the convergence
and symmetry
requirements that will emerge in the definition of our poset. 

\begin{definition}
 An $M$-acceptable sequence\label{Melement} 
 $\langle y^M(s) ~:~ s\in S_\delta\rangle$ ($M\cap \omega_1 = \delta$) 
  is $(M,\omega)$-acceptable
 providing there is a countable set $T\subset \omega_2$
and a $\gamma > \delta = M\cap \omega_1$ 
 such that, for each $s\in  S_\gamma$, 
\begin{enumerate}
\item 
 $s$ forces an ordinal value $x$ on  $e_g(y^M(g\cap S_\delta))$,
\item  for each $\xi \in T$,
 $s$ forces a value, denoted $U_s(x,\xi)$, on $\dot U(x,\xi)\cap M$
  which is a subset of $X[{s\oplus g}]$ (which is a set of ordinals),
\item $s$ forces that $x$  is the only point 
 that is in the closure  of each of the sets
  $$ {\Pi}^X_{\langle s\oplus g\rangle}
\left((\Pi^X_{\langle s_i ~:~ i<n\rangle})^{-1}
\left(\strut\mathbf{\Pi} 
 \{ U_{s_i\oplus s'}(\xi_i) ~:~i<n\}
\cap    e_{s}[F\cap M] \right)\right)$$  
where\label{three}
  $\langle\xi_i~:~ i<n\rangle\in T^n$,  $s'\in S_\gamma$, 
$s'\restriction\delta=s\restriction\delta$,
and $s$ forces that
$F$ is a member of $M\cap \mathcal F_{\langle s_i~:~i<n\rangle}$ for
some $\langle s_i ~:~ i\in n\rangle \in  S^{<\omega}\cap M$. 
Note that $s$ does force a value on $e_g[F\cap M]$,
and so this value is $e_s[F\cap M]$.
\end{enumerate}
\end{definition}

It may help to unravel condition (\ref{three}) a little. The set
 $e_s[F\cap M]$ is the ordinal evaluation of a set that $s$ has forced
 to be in  our filter. 
Then we use an open set intersected with $M$,
 namely $
\strut\mathbf{\Pi} 
 \{ U_{s_i\oplus s'}(\xi_i) ~:~i<n\} $, that comes from sets
that $s'$ forces are in the topology. The connection here is
that $e_s[F\cap M]$ will be equal to $e_{s'}[F'\cap M]$ for some other $F'$
forced by $s'$ to be in the filter. Then the map
$(\Pi^X_{\langle s_i ~:~ i<n\rangle})^{-1}$ pulls this intersection back
into the full product  $\mathbf{\Pi}\{ X[{s\oplus g}] ~:~ s\in S\}$ 
and, finally, 
$ {\Pi}^X_{\langle s\oplus g\rangle}$ is just the projection
to the single coordinate $s\oplus g$. Now we are
asking that $s$ will force
that $x$ will (still)  be the limit even though 
$s'$ may well have assigned a different and disjoint set of 
ordinals to  be limits of any of these sets. This aspect is necessary
for the poset  to be proper and is already handled (can be shown to be)
by
$M$-acceptability. The additional requirement that $s$ forces that
 $x$ is the only such limit is to ensure we get a copy of $\omega_1$. 

\begin{lemma} 
For each countable (suitable)
$M\prec H(\theta)$ (i.e. with $\dot{\mathcal F}_1, \dot X, S$ all
in $M$) there is an\label{firstcountable}
$(M,\omega)$-acceptable sequence.
\end{lemma}

We postpone the proof.  But henceforth we adopt the notation
that for a countable suitable
 $M\prec H(\theta)$, the sequence
 $\langle y^M(s) ~:~ s\in S_{M\cap \omega_1}\rangle$ denotes
the $\prec$-least $(M,\omega)$-suitable sequence. Also
let $\{ T_M (n) ~:~ n\in \omega\}$
 denote a countable subset of $\omega_2$ witnessing
that the sequence is $(M,\omega)$-acceptable, and not just
 $M$-acceptable. We can arrange that $\{ T_M (n) ~:~ n\in \omega\}$
 is forced to form a regular
 filter base.
\bigskip

Now we are ready to define our poset $\mathcal P$.  
Another change we make from \cite{DT}
is that we no longer have  the assumption of countable
bases of neighborhoods, and so we use finite subsets
of $\omega_2$ rather than of $\omega$ as  side conditions.
We will be able to weaken the character assumption 
(through  Lemma \ref{acceptable}) 
by an appeal to  \cite[8.5]{To} which showed
 that there will be a rich supply of
 relative $G_\delta$-points.

\begin{definition}
 A condition $p\in \mathcal P$ consists of $({\mathcal M}_p, S_p, H_p)$
 where $\mathcal M_p$ is a finite $\in$-chain of countable
suitable elementary 
submodels of  $H(\theta)$ and $H_p$ is a finite subset of $\omega_2$.
 We let $M_p$ denote the maximal
 element of $ {\mathcal M}_p$ and let
 $\delta_p = M_p\cap \omega_1$. We require that 
 $S_p$ is a finite subset of $S_{\delta_p}$. 
 For $s\in S_p$
  and $M\in \mathcal M_p$, we use both 
$s\restriction M$ and $s\cap M$ to
  denote $s\restriction (M\cap \omega_1)$.
Note that  the sequence $\{ y^M(s) ~:~ s\in S_{M\cap \omega_1}\}$
is in each $M'$
 whenever $M\in M'$ are both in $\mathcal M_p$.

 It is helpful to simultaneously think of $p$ as inducing a finite
 subtree,  $S_p^\downarrow$,
 of $S$ equal to $\{ s\restriction M ~:~ s\in S_p, \ \mbox{and}\ 
  M\in \mathcal M_p\}$.

  For each $s\in S_p$ and each $M\in {\mathcal M}_p\setminus M_p$
   we define an $S$-name
   $\dot W_p(s\restriction M)$ of a neighborhood of
   $e_s(y^M(s\restriction M))$. 
   It is defined as the name of the
 intersection of all sets of the form 
$\dot U(e_{s'}(y^{M'}(s'\restriction M')),\xi)$ 
   where $s'\in S_p$, $\xi\in H_p$, 
$M'\in \mathcal M_p\cap M_p$, and
$s\restriction M\subset s'\restriction M'$
    and $e_{s'}(
    y^M(s\restriction M)) \in \dot U(e_{s'}(y^{M'}(s'\restriction
    M')),\xi)$. We adopt 
   the convention that $\dot W_p(s\cap M)$ is all of $X$ if
 $s\cap M\notin S_p^\downarrow$.
  \medskip
  
  The definition of $p< q$ is that $\mathcal M_q\subset \mathcal M_p$,
  $H_q\subset H_p$,  
   $S_q \subset S_p^\downarrow$, 
and   for   each $s'\in S_p$    and 
  $s\in S_q$ below $s'$,
   we have that $s'$ forces that
   $e( y^M(s\restriction M))\in \dot W_q(s\restriction M')$
   whenever $M\in \mathcal M_p\setminus \mathcal M_q$
and   $M'$ is the minimal member of $M_q\cap (\mathcal M_q\setminus
M)$.
Another trivial change that we list separately for emphasis is that
for each $M\in \mathcal M_q$ and each $k<|H_q|$, we require
that $T_M(k) \in H_p$.
\end{definition}

It is proven in \cite{DT} that a version of $\mathcal P$ is proper
and  $S$-preserving.  The superficial 
distinction between the two posets can be handled in either of two
ways. We can re-enumerate each $\{ \dot U(x,n) ~:~ n\in \omega\}$ 
so as to be an enumeration of the sequence $\{ \dot U(x,T_M(n)) ~:~ 
n\in \omega\}$, or we can examine the proof in \cite{DT} and
 notice that the proof did not in any way use that we were restricting
 to countably many neighborhoods of each point. 
\bigskip

Before we prove Lemma \ref{firstcountable}, 
 let us prove that this works.

\begin{lemma}
If $\mathcal P$ is proper and $S$-preserving, then PFA(S) implies that
 $S$ forces\label{3.9}
  that $\dot X$
 contains  a copy of $\omega_1$.
\end{lemma}

\begin{proof}
For any condition $q\in \mathcal P$, let $\mathcal M(q)$ denote the
collection of all $M$ such that there exists a $p<q$ such 
that $M\in \mathcal M_p$. 
For each $\beta < \alpha\in \omega_1$, $s\in S_\alpha$, $m\in \omega$,
and $\xi\in \omega_2$, let
\begin{align*}
D(\beta,\alpha,s,\xi,m)  =& \{ p\in \mathcal P ~:~ \xi \in H_p\ ,
 |H_p|\geq m, 
(\exists \bar s\in S_p )~~ 
s < \bar s,  \mbox{and}  \\
& (\exists M\in \mathcal M_p)~ (\beta\in M, \alpha\notin M)\ \  \mbox{or}\\ 
& (\forall M\in \mathcal M(p)) (\beta\in M \Rightarrow \alpha\in M)\}~.
\end{align*}
It is easily shown that each $D(\beta,\alpha,s,\xi,m)$ is a dense open
subset of $\mathcal P$.
Consider the family 
 $\mathcal D$ of all such $D(\beta,\alpha,s,\xi,m)$,
where  $\beta,\alpha,\xi\in \omega_1$ and $m\in \omega$,
and let $G$ be a $\mathcal D$-generic filter.
Let  $\mathcal M_G = \{ M ~:~ (\exists p\in G)~  M \in
 \mathcal M_{p}\}$ 
and let $C = \{ M\cap \omega_1 ~:~ M\in \mathcal M_G\}$.
For each $\delta\in C$, let $M_\delta$ denote the member
of $\mathcal M_G$ such that $M_\delta\cap\omega_1=\delta$
(we omit the trivial
proof that there is exactly one such $M$ for each $\gamma\in C$). 
 Let $g\subset S$ be a
 generic filter and for each $\delta\in C$, let $s_\delta$ be the
 element of $g\cap S_\delta$. Also for each $\delta\in C$,
 let $x_\delta = e_g(y^{M_\delta}(s_\delta))$.  
Let us also note that for any $\beta<\delta$ both in $C$,
 $x_\beta$  is equal to  $e_{s_\delta}(y^{M_\beta}(s_\beta))$
because of the fact that $y^{M_\beta}$ is an element of
 $M_\delta$.

 We show that the set 
 $W = \{ x_\gamma ~:~ \gamma \in C\}$ 
 is homeomorphic to the ordinal $\omega_1$. 
Indeed, the map $f=\dot f[g]$
is a homeomorphism onto the cub $C$.
It is certainly 1-to-1 and onto. Let us show that it is a closed map.
Let $\{ \delta_\ell ~:~ \ell\in \omega\}\subset C$
 be  strictly increasing with supremum $\delta$. 
We simply have to prove that $x_\delta$ is a limit
of the sequence $\{ x_{\delta_\ell} ~:~ \ell\in \omega\}$. To do so,
 by Definition \ref{Melement}, it suffices to
prove that if  $\gamma\in C\setminus \delta+1$, $T$,
  $\langle\xi_i~:~ i<n\rangle\in T^n$,  $s'\in S_\gamma$, 
$s'\restriction\delta=s_\delta$,
and $s_\gamma$ forces that 
$\dot F$ is a member of $M_\delta\cap \mathcal F_{\langle s_i~:~i<n\rangle}$ for
some $\langle s_i ~:~ i\in n\rangle \in  S^{<\omega}\cap M_\delta$
(as in condition (\ref{three})), then 
 $\{ x_{\delta_\ell} ~:~ \ell\in \omega\}$
meets
  $$ {\Pi}^X_{\langle s_\gamma\oplus g\rangle}
\left((\Pi^X_{\langle s_i ~:~ i<n\rangle})^{-1}
\left(\strut\mathbf{\Pi} 
 \{ U_{s_i\oplus s'}(\xi_i) ~:~i<n\}
\cap    e_{s_\gamma}[F\cap M] \right)\right)$$

 Fix any $m\in \omega$ large enough so that $\{\xi_i ~:~ i<n\}\subset
\{ T(\ell) ~:~ \ell < m\}$ and
choose
 $p\in G\cap D(0,\delta,s_\gamma,m,m)$ so that
  $\{M_\delta,M_\gamma\} \subset \mathcal M_p$.
Let $\beta$ be the maximum element
 of $C\cap \delta$ such
that $ M_\beta \in \mathcal M_p$. 
 By strengthening $p$ and possibly raising $\beta$,
 we
can assume that $\{ s_i~:~i<n\},\dot F$ are in 
$M_\beta$, and that $\langle s_\beta, \{s_i~:~ i<n\},\dot F\rangle$
is in $\mathcal F$.
Choose  $\ell_0\geq m$ such that 
 $\delta_\ell > \beta$ for all
 $\ell>\ell_0$. 
Additionally choose  $r<p$ with $r\in G$
so that $s'\in S_r^\downarrow$, and for each $i<n$,
 $\{ s_i\oplus s_\gamma, s_i\oplus s'\} \subset S_r^\downarrow$.
Choose $\ell$ large enough so that $M\cap \omega_1 <\delta_\ell$
for all $M\cap \mathcal M_{r}\cap M_\delta$.

\bigskip

We note, in $V[g]$,
 that $x_{\delta_\ell} =
e_{s_\delta}(y^{M_{\delta_\ell}}(s_0\oplus s_{\delta_\ell}) )$
and  $e_{s_\delta}(y^{M_{\delta_\ell}}(s_0\oplus s_{\delta_\ell}) )$
is clearly an element of 
$\Pi^X_{\langle s_\gamma\oplus g\rangle}(
(\Pi^X_{\langle s_i\oplus g ~:~ i<n\rangle})^{-1}(
\langle e_{s_{\delta}}
(y^{M_{\delta_\ell}}(s_i\oplus s_{\delta_\ell})) ~:~
i<n\rangle))$. Therefore it suffices to prove
that
$\langle e_{s_{\delta}}(y^{M_{\delta_\ell}}(s_i\oplus s_{\delta_\ell})) ~:~
i<n\rangle$ is in each of
$\mathbf{\Pi} 
 \{ U_{s_i\oplus s'}(\xi_i) ~:~i<n\}$
and $   e_{s_\gamma}[F\cap M]$.
Choose any $\ell'>\ell$. We note that
$\langle e_{s_{\delta}}(y^{M_{\delta_\ell}}(s_i\oplus
s_{\delta_\ell})) ~:~ i<n\rangle$
 is equal to 
$\langle e_{s_{\delta_{\ell'}}}(y^{M_{\delta_\ell}}(s_i\oplus
s_{\delta_\ell})) ~:~ i<n\rangle$.
It  follows from the definition of each $\dot W_r$
that $s_i\oplus s'$ 
forces that $\dot W_r(s_i\oplus s_\delta)\cap M_\delta
\subset U_{s_i\oplus s'}(\xi_i)$ for each $i<n$.
It therefore follows, from the ordering on $\mathcal P$,
that  
$\langle e_{s_{\delta_{\ell'}}}(y^{M_{\delta_\ell}}
(s_i\oplus s_{\delta_\ell}))~:~i<n\rangle\in 
\mathbf{\Pi}\{ U_{s_i\oplus s'}(\xi_i) ~:~ i<n\}$.
Finally, since
$\langle y^{M_{\delta_\ell}}(\bar s) ~:~ \bar s\in S_{\delta_\ell} \rangle$ 
is $M_{\delta_\ell}$-acceptable, we have that
$\langle e_{s_{\delta_{\ell'}}}(y^{M_{\delta_\ell}}(s_i\oplus s_{\delta_\ell})) ~:~
i<n\rangle$ is
a member of $e(B^{(\delta_\ell+1)})$ for some $B$ that is forced by
$s_\delta$ to be a subset of $\dot F$.
That is, we have that 
$\langle e_{s_{\delta_{\ell'}}}(y^{M_{\delta_\ell}}(s_i\oplus s_{\delta_\ell})) ~:~
i<n\rangle$ is a member of
$ e_{s_\gamma}[F\cap M]$.
\end{proof}

Now we prove Lemma \ref{firstcountable}.

\bgroup

\def\proofname{Proof of Lemma \ref{firstcountable}.}
  
\begin{proof}
Fix $M$ etc. as in the hypothesis of the Lemma.
For each generic filter $g$, let $X[g](\delta)$ denote the compact
pre-image of $[0,\delta]$ by the function $\dot f$. 
Let us note that
the  countable product space $\Pi \{ X[{s\oplus g_0}](\delta) ~:~ s\in
S_\delta\}$  is compact and  sequential (\cite[2.5]{IsmailCclosed});
 and therefore
we know from \cite[8.5]{To} that 
every closed subset has a relative $G_\delta$-point.

Let $\{s_{\delta,j} ~:~ j\in \omega\}$ be the
 $\prec$-least enumeration of
 $S_\delta$ ($\delta = 
 M\cap \omega_1$). We will also use   $s_\delta$ to
denote $ s_{\delta,0}$.
We establish some notation. Given any generic $g\subset S$,
and working  in the extension $V[g]$, 
let $\Pi_{M,g}$ denote the projection map from  
the product  $\mathbf{\Pi} \{ X[{s\oplus g}] ~:~ s\in S\}$ onto
the countable product
  $\mathbf{\Pi} \{ X[{s_{\delta,j}\oplus g}] ~:~ j\in \omega\}$.
Let us note that, for each $j$,  $X[{s_{\delta,j}\oplus g}]\cap M$
as a set  is simply equal to $\omega_2 \cap M$. Now
let $\Pi_g$ denote the canonical isomorphism from 
  $\mathbf{\Pi} \{ X[{s_{\delta,j}\oplus g}] \cap M~:~ j\in \omega\}$
to $\left( \omega_2\cap M\right)^\omega$ which is induced by
the bijection on the indices sending $s_{\delta,j}$ to $j$.

We adopt one more notational convention. Consider a vector
 $\vec x$ that is a function from $S_\delta$ to $\omega_2$,
i.e. $\vec x\in \omega_2^{S_\delta}$. When given
 any generic filter $g$, we will regard $\vec x$
(without using any notational device) as also being a function from
 $\{ s\oplus g ~:~ s\in S_\delta\}$ into $\omega_2$. Thus, given any
 $g$, we see $\vec x$ as representing a point in $\mathbf{\Pi}\{
 X[{s_{\delta, j}\oplus g}] ~:~ j\in \omega\}$.

Now fix any generic $g_0$ with $s_\delta\in g_0$.
Recall the notation that $\mathcal F$ is the enumerated
family $\{  (s^\alpha, \{ s^\alpha_i ~:~ i < n_\alpha\}, 
 \dot F_\alpha) ~:~ \alpha\in \lambda\}$. 
Let $\Lambda_\delta = \{ \alpha \in M\cap \lambda ~:~ s^\alpha < s_\delta\}$.
We have that, for each $\alpha\in \Lambda_\delta$,
$s_\delta$ forces a value on  
 $\dot F_\alpha\cap M \subset  Y^n$. It follows also that
 $s_\delta$ forces a value
 on ${\tilde \Pi }^{-1}_{\langle s^\alpha_i ~:~ i<
   n\rangle}(\dot F_\alpha\cap M)$.
Let  $e_\delta(\dot F_\alpha)$ denote the resulting subset 
 $\Pi_{M,g_0}\left( 
{\tilde \Pi }^{-1}_{\langle  s^\alpha_i ~:~ i<
   n\rangle}(\dot F\cap M)\right)$
of $\mathbf{\Pi} \{ X[{s_{\delta,j}\oplus g_0}] ~:~ j< \omega\}$. 
Let  
 $\mathcal H_0$ denote the countable collection, which is in $V$, 
$\{ \Pi_g[e_\delta(\dot F_\alpha) ] ~:~ \alpha\in \Lambda_\delta\}$,
of subsets of $\left(\omega_2\cap M\right)^\omega$.

Before continuing, let us observe that for each
 $\beta \in M\cap \lambda$ and each $j\in \omega$
such that $s^\beta < s_{\delta,j}$, there is an $\alpha \in 
\Lambda_\delta$ such that $s_\delta$ forces that
$ \dot F_\alpha \cap M$ 
is equal to the same set that $s_{\delta,j}$ forces
that $\dot F_\beta \cap M$ to be. This is by property
 (4) of Theorem \ref{filterF} and the fact that
 $S$ is a coherent tree. Indeed, it is immediate
that for each $\xi\in\delta $ with $o(s^\beta) <\xi$,
 we also have that $(s_{\delta,j}\restriction \xi , \{s_i^\beta ~:~ i<
 n_\beta \}, \dot F_\beta)$ is in the list
 $\{ (s^\alpha, \{s_i^\alpha ~:~ i< n_\alpha \}, \dot F_\beta ) ~:~ \alpha
 \in M\cap \lambda\}$. Therefore, by increasing $o(s^\beta)$ 
 we can assume that $s^\beta\oplus s_\delta$ is equal to
 $s^\delta_j$. Then property (4) of Theorem \ref{filterF}
says that there is an $\alpha\in \Lambda_\delta$ so that
$(s^\alpha, \{ s^\alpha_i ~:~ i< n_\alpha \} , \dot F_\alpha)$
satisfies that $\{ s^\alpha_i ~:~ i< n_\alpha \}$ is
equal to  $ \{ s^\beta_i ~:~ i < n_\alpha \}$
and $\dot F_\alpha$ is equal to $(\dot
F_\beta)^{s^\beta}_{s^\alpha}$. 
What the current paragraph has shown is that if we repeat
this same process with any other generic $g$ and any other
member of $S_\delta$ (in place of $s_\delta$), then we still 
end up with the same collection $\mathcal H_0$.
Stated another way, each $s\in S\setminus M$ forces
that, letting $g$ denote the generic,
the collection 
$\{ \left( \Pi_{M,g}\circ \Pi_g\right)^{-1}(H)
 ~:~ H\in \mathcal H_0\}$ is equal to the
collection of all sets of the form
 $(\Pi^X_{\langle s_i ~:~ i< n\rangle})^{-1} (e_s(M\cap \dot F))$ 
that we have to consider in the statement of our Lemma.

Continuing in $V[g_0]$, we  work in
the   space
 $\mathbf{\Pi} \{ X[{s\oplus g_0}](\delta) ~:~ s\in S_\delta\}$.
Since $\Pi_{g_0}^{-1}[\mathcal H_0]$  has the finite intersection
property,  
and  letting $K_0$ denote the intersection of the closures,
we may  choose some
 $\vec x_0 \in \omega_2^{S_\delta}$ (following our convention)
in $K_0$
with relative countable character. Again note that $\vec x_0$ is
actually a sequence of ordinals in $\omega_2$ in that, for 
each $s\in S_{\delta}$, $\vec x_0(s\oplus g_0)\in \omega_2$.
Choose a countable
 subset 
 $T_0\subset \omega_2$ so as to generate a relative neighborhood base
 for the point. Finally, choose $s_0\in g_0$ (with $s_\delta\leq s_0$) 
so that $s_0$ forces all these properties. Here is how much progress
we have made in the proof~:~

\begin{claim}
For each $j\in \omega$,  each extension $s$ of
$s_{\delta,j}\oplus s_0$ forces that $\vec x_0(s_{\delta,j})$
  is the
only point  
 that is in the closure  of each of the sets
  $$ {\Pi}^X_{\langle s\oplus g\rangle}\left((\Pi^X_{\langle s_i ~:~
      i<n\rangle})^{-1}\left( 
\strut\mathbf{\Pi}
 \{ U_{s_i\oplus s'}(\xi_i) ~:~i<n\}
\cap    e_{s}[ M\cap F] \right)\right)$$  
where  $\langle\xi_i~:~ i<n\rangle\in T_0^n$, 
$s'\supset s_{\delta,j}$ is  sufficiently large,
and $s_0$ forces that
$F$ a member of $M\cap \mathcal F_{\langle s_i~:~i<n\rangle}$ for
some $\langle s_i ~:~ i\in n\rangle \in  S^{<\omega}\cap M$. 
\end{claim}

Thus, so long as our final choice for the countable set
 $T$ contains $T_0$ and our choice for $\{ y^M(s) ~:~ s\in S_\delta\}$  
is forced by $s_0$ to satisfy that
 $e(y^M(s_{\delta,j})) = \vec x_0(s_{\delta,j}) $ for each $j\in \omega$,
then
we will have satisfied condition
 (3) for each $s$ that extends one of the elements
of the antichain $\{s_{\delta,j} \oplus s_0 ~:~ j\in \omega\}$.  

This was the first step of an induction. We enlarge the family
$\mathcal H_0$ to the family $\mathcal H_1$ that also has the finite
intersection property. 
Namely, for each  $n\in \omega$, 
and each $\langle \xi_i ~:~ i  < n\rangle \in T_0^n$,
the set
 $$\Pi_{g_0}\left(\mathbf{\Pi}\{  U_{s_{\delta,i}\oplus s_0}(\xi_i) ~:~i<n\}
\times \mathbf{\Pi}\{ M\cap X_{s_{\delta,i}\oplus g_0} ~:~ n\leq i\in
\omega\}\right)$$
is also in $\mathcal H_1$.  Each member of $\mathcal H_1$
is a subset of $\left(\omega_2\cap M\right)^\omega$. Let $g_1$ be any
other generic with $s_0\notin g_1$. Replacing $\mathcal H_0$ in the
above
argument by $\mathcal H_1$, we work in the space
$\mathbf{\Pi} \{ X_{s\oplus g_1}(\delta) ~:~ s\in S_\delta\}$,
 and let $K_1$ be the intersection of the closures of all members
of $\Pi_{g_1}^{-1}[\mathcal H_1]$ 
and we again choose a point $\vec x_1\in K_1$ with relative countable
character. Next, choose a countable set $T_0 \subset T_1\subset
\omega_2$ witnessing that $\vec x_1$ has countable character in
$K_1$. Choose any $s_1\in g_1$ (incomparable with $s_0$) that forces
all of the above properties. Again expand $\mathcal H_1$ to $\mathcal
H_2$ by adding all sets of the form
 $$\Pi_{g_1}\left(\mathbf{\Pi}\{  U_{s_{\delta,i}\oplus s_1}(\xi_i) ~:~i<n\}
\times \mathbf{\Pi}\{ M\cap X_{s_{\delta,i}\oplus g_1} ~:~ n\leq i\in
\omega\}\right)$$  where 
  $n\in \omega$ and
 $\langle \xi_i ~:~ i  < n\rangle \in T_1^n$. It is immediate by the
 construction that the analogue of Claim 1
 holds when $s_0$ is replaced by $s_1$ and $\vec x_0$ is replaced by
 $\vec x_1$. In fact, we formulate a stronger statement 
 for our  inductive hypothesis.  Suppose that $\beta\in \omega_1$
and that we have chosen $\{ s_\alpha ~:~ \alpha< \beta \}$,
  $\{ T_\alpha ~:~ \alpha< \beta  \}$, $\{ \mathcal H_\alpha ~:~
  \alpha<\beta\}$, 
and $\{ \vec x_\alpha  ~:~ \alpha<  \beta\}$   such that\\[5pt]
 $\mathbf{IH}_\beta$~:~ for each $\alpha<\beta$ and each $j\in \omega$

\begin{enumerate}
\item $\{ s_\alpha ~:~ \alpha < \beta\} \subset S\setminus M$ is an
  antichain, 
\item $\{ T_\alpha ~:~ \alpha\in\beta\}$ is an increasing chain of
  countable subsets of $\omega_2$,
\item $\{ \mathcal H_\alpha ~:~ \alpha<\beta\}$ is an increasing chain
and
 $\mathcal H_\alpha$ is a family of subsets of $
\left(M\cap \omega_2\right)^\omega$ that has the finite intersection
property, 
\item  $\{ \vec x_\alpha ~:~ \alpha < \beta \}\subset
  \omega_2^{S_\delta}$
\item  $s_\alpha$ forces that $\vec x_\alpha$ is a point in
 $\mathbf{\Pi}\{ X_{s_{\delta,i}\oplus g}(\delta) ~:~ i\in \omega\}$
 \item for each $\xi\in T_\alpha$, 
  $s_{\delta,j}\oplus s_\alpha$ forces a value, $U_{s_{\delta,j}\oplus
    s_\alpha}$, 
 on $\dot U(\vec x_\alpha(s_{\delta,j}), \xi)\cap M$,
\item 
 $\mathcal H_\alpha$ is equal to $\bigcup_{\eta<\alpha} \mathcal
 H_\eta$ together with all sets of the form
 $$\Pi_{g_\alpha}\left(\mathbf{\Pi}\{  U_{s_{\delta,i}\oplus
     s_\alpha}(\xi_i) ~:~i<n\} 
\times \mathbf{\Pi}\{ M\cap X_{s_{\delta,i}\oplus g_\alpha} ~:~ n\leq
i\in \omega\}\right)$$  where
  $n\in \omega$, 
 $\langle \xi_i ~:~ i  < n\rangle \in T_\alpha^n$, 
and $g_\alpha$ is any generic with $s_\alpha\in g_\alpha$,
\item for any $\alpha \leq \eta<\beta$,
$s_\alpha$ forces that
 $\vec x_\alpha$  is the
only point  
 that is in the closure  of each of the sets in the collection
$\{  \Pi_{g_\alpha}^{-1}(H)
 ~:~ H\in \mathcal H_\eta\}$~.
\end{enumerate}

Having defined the family as above, we stop the construction if $\{
s_\alpha ~:~ \alpha \in \beta\}$ is a maximal antichain. Otherwise, we
choose a generic $g_\beta$ disjoint from $\{ s_\alpha ~:~ \alpha <
\beta\}$ and work in $V[g_\beta]$. The construction of 
 $\vec x_\beta$, $T_\beta\supset \bigcup \{ T_\alpha ~:~
 \alpha<\beta\}$,
and $\mathcal H_\beta$ proceeds as in the selection for $\beta = 1$. 
For any $\alpha<\beta$, item (8) appears to be getting stronger but
only in the sense that we are preserving that $\vec x_\alpha$ is in
the closure of $\Pi_{g_\alpha}^{-1}(H)$ for $H\in \mathcal
H_\beta$. The fact that $\mathcal H_\beta$ has the finite intersection
property guarantees that there is some point in all of the closures,
and by the uniqueness guaranteed by 
 $\mathbf{IH}_\beta$ we have that $\vec x_\alpha$ is that point.
\bigskip

Now we assume that $\{ s_\alpha ~:~ \alpha < \beta \}$ is a maximal
antichain. For the statement of the Lemma we can let $\gamma$ be any
large enough ordinal such that $o(s_\alpha) \leq \gamma$ for all 
 $\alpha< \beta$. We want to define our $M$-acceptable sequence
 $\langle y^M(s) ~:~ s\in S_\delta\rangle$ so that each $s_\alpha$
 forces $\vec x_\alpha$ to equal $\langle e(y^M(s)) ~:~ s\in
 S_\delta\rangle$. 
  We skip the proof that $\mathbf{IH}_\beta$ ensures
 that  if we succeed, then this is also 
 an $(M,\omega)$-acceptable sequence. 
Let $\mathcal H_\beta$ be the filter base generated by 
all finite intersections from the family
 $\bigcup \{ \mathcal H_\alpha ~:~ \alpha < \beta \}$. $\mathcal
 H_\beta$ is a filter base on the product set
 $\left(M\cap \omega_2\right)^\omega$. Let us note that for each $H\in 
 \mathcal H_\beta$, there is an $j= j_H\in \omega$ such that $H$ can be
 factored as $\mathbf{\Pi}\{ H(i) ~:~ i<j\} \times
 (M\cap \omega_2)^{\omega\setminus j}$.  Since $\mathbf{\Pi}\{H(i) ~:~
 i<j_H\}$ is a subset of $M$, it then  follows that 
$H\cap M$ is not empty for each $H\in \mathcal H_\beta$. 

Let $\{\alpha_k ~:~ k\in \omega\}$ be an enumeration of
 $\Lambda_\delta$. Let $\{ H_m ~:~ m\in \omega\}$ enumerate
a descending base for 
$\mathcal H_\beta$.
Choose the sequence $\{ H_m ~:~ m\in \omega\}$ so that
the element 
$\Pi_g[e_\delta(\dot F_{\alpha_m}) ]$ of $\mathcal H_0$ contains
$H_m$. 
   For each $m$  let $j_m \geq m, j_{H_m}$. 
For each $m$, choose $\vec \xi_m \in M\cap H_m$. 
Choose any $\eta_m\in M\cap \delta$ large enough so that for
each $i<j_m$, $(s_{\delta,i}\restriction \eta_m)\oplus s_\delta$ 
is equal to $s_{\delta,i}$. If needed, we can increase $\eta_m$
so that, for each $i<j_m$,
 $s_{\delta,i}$ forces that the ordinal
 $\vec \xi_m(i)$ is in the sequential closure
of $\eta_m\subset \omega_1$.  Since $s_\delta$ forces that
 $\mathbf {\Pi} \{ X_{s_{\delta,i} \oplus g} (\eta_m)~:~ i < j_m\}$ is
 sequential, there is an $\bar s_m < s_\delta$ (in $M$) 
and a  $\vec y_m \in Y^{\{s_{\delta,i}\restriction\eta_m
\oplus \bar s_m~:~ i < j_m\}}\cap M$
so that $\bar s_m$ 
  forces that
$e_g(\vec y_m)$ is equal to $\vec \xi_m\restriction j_m$. 
We can, and do, require a little more of each such
 witness $\vec y_m$. 
There are many members $\vec y$ of $\Lambda^{\{ s_{\delta,i} ~:~ i < m\}}$
 that satisfy  that $s_\delta$ forces that $e_g(\vec y)$
is equal to $\vec \xi_m \restriction \eta_m$. It is a simple matter to
merge finitely many of them so as to ensure that for 
 each  $\alpha \in \{\alpha_k ~:~ k\leq m\}$, 
$\vec y_m \restriction \{ s^\alpha_i \oplus \bar s ~:~ i < n_\alpha\}$
is a member of $e_g(\dot F_\alpha)$. This means that if 
$(s^\alpha , \{s^\alpha_i ~:~ i < n_\alpha \}, \dot A)$ is a member of
 $\mathcal A\cap M$ satisfying that $\dot F_\alpha = (\dot
 A)^{(\omega_1)}$, 
(for some $\alpha\in \{\alpha_k ~:~ k\leq m\}$)
then $\vec y_m \restriction \{
 s^\alpha_i \oplus \bar s ~:~ i < n_\alpha\}$ is a member of
 $(\dot A)^{(<\delta)}$.

We are ready to define $y^M(s_{\delta,i})$
as in Definition \ref{seql}. 
It should be clear that each member of the maximal antichain
 $\{ s_\alpha ~:~ \alpha < \beta\}$ forces that 
 $\{ \vec y_m(i) ~:~ i < m \in \omega\}$ is a sequence that converges,
 with $s_\alpha$ forcing that it converges to 
 $\vec x_\alpha(i)$. Thus, we define $y^M(s_{\delta,i})$ as
 being built from  $\{ \vec y_m(i) ~:~ i< m \in \omega\}$.
This shows that $y^M(s_{\delta,i})$ is a member of
 $\Lambda$. Finally we check that the sequence
 $\{ y^M(s) ~:~ s\in S_\delta\}$ is $M$-acceptable. 
Following Definition \ref{acceptable},  we take any
 $(\bar s, \{ \bar s_i ~:~ i<n\} , \dot A)$ in $\mathcal A\cap M$. 
We also 
fix any $j< \omega$ and letting $s_{\delta,j}>\bar s$ be the $s$ in the
statement in Definition \ref{acceptable}. 
Choose any $m>j$ large enough so that $\{\bar s_i \oplus s_\delta  ~:~ i
< n\}$ is a subset of $\{ s_{\delta, i} ~:~ i < j_m\}$. By possibly
increasing $\bar s$, we can assume that there is an $\alpha\in
\Lambda_\delta$  so that
 $s^\alpha_i  = \bar s_i \oplus s^\alpha_0 $ ($i<n_\alpha$)
and   $\dot F_\alpha$ is equal to $\left( \dot A^{\bar
    s}_{s^\alpha}\right)^{(\omega_1)}$. So long as $m$ is also large
enough that this $\alpha$ is in the set $\{ \alpha_k ~:~ k\leq m\}$ we
have that $s_\delta$ forces that $\vec y_m \restriction
 \{s^\alpha_i\oplus g ~:~ i < n_\alpha\}$ is in $\dot F_\alpha$.
 This is
 equivalent to the fact that $s_{\delta,j}$ forces that
 $\vec y_m \restriction \{s^\alpha_i \oplus g ~:~ i < n_\alpha\}$
(which equals $\vec y_m \restriction \{\bar s_i \oplus g ~:~ i <
n_\alpha\}$)  is in $(\dot A)^{(<\delta)}$.  Let $J\in \omega$ be
chosen
 so that each $m>J$ is sufficiently large as above. Well now,
 for each  $m>J$, let
  $\vec b_m\in Y^n$ be simply $\vec y_m\restriction \{\bar s_i \oplus
  s_\delta\} $ with the indices relabelled. Then
 it follows that $s_{\delta,j}$ forces that 
 $\langle y^M(\bar s_i\oplus s_{\delta,j}) ~:~ i<n\rangle$ is
in $B^{(\delta+1)}$ where $B = 
\{  \vec b_m  ~:~  J< m \in \omega\}$ is forced by 
 $s_{\delta,j}$ to be a subset of $\dot A\cap M$. 
\end{proof}

\bibliographystyle{acm}
\bibliography{countablycompact}

\begin{thebibliography}{10}

\bibitem{B}
{\sc Balogh, Z.~T.}
\newblock Locally nice spaces under {M}artin's axiom.
\newblock {\em Comment. Math. Univ. Carolin. 24}, 1 (1983), 63--87.

\bibitem{vD}
{\sc {\noopsort{Douwen}}van~Douwen, E.~K.}
\newblock The integers and topology.
\newblock In {\em Handbook of {S}et-theoretic {T}opology}, K.~Kunen and J.~E.
  Vaughan, Eds. North-Holland, Amsterdam, 1984, pp.~111--167.

\bibitem{D}
{\sc Dow, A.}
\newblock On the consistency of the {M}oore-{M}r\'owka solution.
\newblock {\em Topology Appl. 44\/} (1992), 125--141.

\bibitem{D1}
{\sc Dow, A.}
\newblock Set-theoretic update on topology.
\newblock In {\em Recent {P}rogress in {G}eneral {T}opology, {III}}. Springer,
  2013, pp.~329--357.

\bibitem{DT}
{\sc Dow, A., and Tall, F.~D.}
\newblock Hereditarily normal manifolds of dimension $> 1$ may all be
  metrizable, submitted.

\bibitem{DT1}
{\sc Dow, A., and Tall, F.~D.}
\newblock Normality versus paracompactness in locally compact spaces, in
  preparation.

\bibitem{eisworth2002}
{\sc Eisworth, T.}
\newblock On perfect pre-images of {$\omega_1$}.
\newblock {\em Topology Appl. 125}, 2 (2002), 263--278.

\bibitem{FTT}
{\sc Fischer, A.~J., Tall, F.~D., and Todorcevic, S.}
\newblock Forcing with a coherent {S}ouslin tree and locally countable
  subspaces of countably tight compact spaces.
\newblock {\em Topology Appl. 195\/} (2015), 284--296.

\bibitem{Fr}
{\sc Franklin, S.~P., and Rajagopolan, M.}
\newblock Some examples in topology.
\newblock {\em Trans. Amer. Math. Soc. 155\/} (1970), 305--314.

\bibitem{G}
{\sc Gruenhage, G.}
\newblock Some results on spaces having an orthobase or a base of subinfinite
  rank.
\newblock In {\em Proceedings of the 1977 Topology Conference (Louisiana State
  Univ., Baton Rouge, La., 1977), I\/} (1977), vol.~2, pp.~151--159 (1978).

\bibitem{G2}
{\sc Gruenhage, G.}
\newblock Paracompactness and subparacompactness in perfectly normal locally
  bicompact spaces.
\newblock {\em Usp. Mat. Nauk. 35\/} (1980), 44--49.

\bibitem{IsmailCclosed}
{\sc Ismail, M.}
\newblock Product of {$C$}-closed spaces.
\newblock {\em Houston J. Math. 10}, 2 (1984), 195--199.

\bibitem{LT}
{\sc Larson, P., and Tall, F.~D.}
\newblock Locally compact perfectly normal spaces may all be paracompact.
\newblock {\em Fund. Math. 210\/} (2010), 285--300.

\bibitem{LTo}
{\sc Larson, P., and Todorcevic, S.}
\newblock Kat\v etov's problem.
\newblock {\em Trans. Amer. Math. Soc. 354\/} (2002), 1783--1791.

\bibitem{N}
{\sc Nyikos, P.~J.}
\newblock On first countable, countably compact spaces {III}: {T}he problem of
  obtaining separable noncompact examples.
\newblock In {\em Open {P}roblems in {T}opology}, J.~van Mill and G.~M. Reed,
  Eds. Elsevier, Amsterdam, 1990, pp.~127--161.

\bibitem{Ny3}
{\sc Nyikos, P.~J.}
\newblock Hereditary normality versus countable tightness in countably compact
  spaces.
\newblock In {\em Proceedings of the Symposium on General Topology and
  Applications (Oxford, 1989)\/} (1992), vol.~44, pp.~271--292.

\bibitem{To}
{\sc Todorcevic, S.}
\newblock Forcing with a coherent {S}ouslin tree.
\newblock Preprint.

\bibitem{To1}
{\sc Todorcevic, S.}
\newblock Directed sets and cofinal types.
\newblock {\em Trans. Amer. Math. Soc. 290\/} (1985), 711--723.

\bibitem{W}
{\sc Weiss, W.}
\newblock Countably compact spaces and {M}artin's axiom.
\newblock {\em Canad. J. Math. 30}, 243--249 (1978).

\end{thebibliography}

{\rm Alan Dow, Department of Mathematics and Statistics, University of North Carolina,
Charlotte, North Carolina 28223}

{\it e-mail address:} {\rm adow@uncc.edu}

{\rm Franklin D. Tall, Department of Mathematics, University of Toronto, Toronto, Ontario M5S 2E4, CANADA}

{\it e-mail address:} {\rm f.tall@math.utoronto.ca}

\end{document}